\documentclass[12pt,psfig]{article}
\usepackage{graphics,epsfig,amssymb}
\pdfoutput=1
\usepackage{cite}

\usepackage{amsmath}
\usepackage{amsthm}
\usepackage{amsfonts}
\usepackage{amssymb,latexsym}
\usepackage[all]{xy}
\usepackage{graphicx}
\usepackage[mathscr]{eucal}
\usepackage{verbatim}
\usepackage{hyperref}

\usepackage{txfonts}

\usepackage{mathrsfs}
\usepackage{color}
\usepackage{hyperref}

\theoremstyle{plain}
    \newtheorem{thm}{Theorem}[section]
    \newtheorem{prop}{Proposition}[section]
    \newtheorem{lemma}{Lemma}[section]

\numberwithin{equation}{section}

\usepackage{graphicx}
\usepackage{pict2e}
\usepackage{slashbox}

\begin{document}

\title{Existence and uniqueness of  ground state  solutions for the planar Schr\"odinger-Newton equation on the disc \thanks{The first author was partially supported by  Scientific Research Fund of Hunan Provincial Education Department (Grant No. 22B0484) and Natural Science Foundation of Hunan Province (Grant No. 2024JJ5214); the third author  was partially supported by Natural Science Foundation of Hunan Province (Grant No. 2022JJ30235).}}

\author{
Hui Guo\\
  Department of Mathematics and Finance,\\
Hunan University of Humanities, Science and Technology, \\
Loudi, Hunan 417000, P. R. China\\
Email: huiguo\_math@163.com;\\
 Zhiwen Long\\
   Department of Mathematics and Finance,\\
Hunan University of Humanities, Science and Technology, \\
Loudi, Hunan 417000, P. R. China\\
Email:longzw2005@126.com;
\\
Tao Wang\\
  College of Mathematics and Computing Science,\\
Hunan University of Science and Technology, \\
Xiangtan, Hunan 411201, P. R. China\\
Email: wt\_61003@163.com\\}
\date {}
\maketitle

\begin{abstract}
This paper is concerned with the existence and qualitative properties of   positive ground state solutions for the planar Schr\"odinger-Newton equation on the disc. First, we prove the existence and radial symmetry of all the positive ground state solutions by employing  the symmetric decreasing rearrangement and  Talenti's inequality. Next, we develop  Newton's theorem and then use the contraction mapping principle to establish the uniqueness of the  positive ground state solution for the  Schr\"odinger-Newton equation on the disc in the two  dimensional case. Finally, we show that the unique positive ground state solution converges to the trivial solution as the radius $R$ tending to infinity, which is totally different from the higher dimensional case in \cite{Guo-Wang-Yi}.
 \end{abstract}

\section{Introduction}
In the  recent twenty years, the classical stationary  Choquard  equation in the whole space $\mathbb{R}^N$ with a higher dimension $N\geq 3,$
\begin{equation}\label{1.1}
-\Delta u+u=\left(\int_{\mathbb{R}^N}\frac{|u(y)|^2}{|x-y|^{N-2}}dy\right)u\quad \mbox{in}\,\ \mathbb{R}^N,\hfill\\
\end{equation}
have been  widely studied by many researchers. This type of equations arises in several physical fields. It was first proposed in 1954 by Pekar in describing the quantum mechanics of polaron at rest. Later, in 1976, Choquard used it to model an electron trapped in its own hole, in a certain approximation to Hartree-Fock theory of one-component plasma (see e.g.  \cite{Lieb,Lieb-Simon,Palais} for more details). In mathematical contents, the existence and qualitative properties of solutions for \eqref{1.1} have been investigated  intensively in the literature. For more related work, one can refer to \cite{Feliciangeli-Seiringer,Ma-Zhao,Moroz-Van Schaftingen,clapp,Ruiz-Van Schaftingen,Wang-Yi1,Wang-Yi2,xcl} and references therein.

As we know,  problem \eqref{1.1} is equivalent to
\begin{equation}\label{model3}
\left\{
\begin{array}{lll}
&-\Delta u+u=wu\ &\mbox{in }\mathbb{R}^N,\\
&-\Delta w=|u|^2\ &\mbox{in }\mathbb{R}^N.
\end{array}\right.
\end{equation}
When $N=2$, by using Green's function, \eqref{model3}  turns into
 the following  planar Choquard equation ( also called Schr\"odinger-Newton equation)
\begin{equation}\label{1.2}
-\Delta u+u=\left(\int_{\mathbb{R}^N}\left(\ln\frac{1}{|x-y|}\right)|u(y)|^2 dy\right)u\quad \mbox{in}\,\ \mathbb{R}^2.
\end{equation}
For its physical background and details, one can refer to \cite{Cingolani-Weth} and references therein. Compared with  the higher dimensional case $N\geq3$,  the dimension $N=2$ is a critical number for the Sobolev embedding theorem and the logarithmic kernel in \eqref{1.2} is unbounded both from above and below. This  together with the fact that it
is sign-changing introduces a major difficulty.
So  the two dimensional case seems much more complex.
In order to prove the existence of solutions for \eqref{1.2},
Stubbe\cite{Stubbe}  applied the classical variational methods to \eqref{1.2} and  established the existence of   a unique ground state solution which is a positive and spherically symmetric
decreasing function by introducing a suitable working space
$$
X:=\left\{u\in H^1(\mathbb{R}^2)\mid\int_{\mathbb{R}^2}\ln(1+|x|)u^2\operatorname{d}x<+\infty\right\},
$$
which is a new subspace constrained in $H^{1}(\mathbb{R}^2).$  Later,  Cassani and Tarsi \cite{Cassani-Tarsi} provided a proper function space setting  by a new weighted version of the Pohozaev-Trudinger inequality which enables   to prove the existence of nontrivial finite energy solutions to \eqref{1.2}.  Liu, Radulescu, Tang and Zhang \cite{Liu-Radulescu-Tang-Zhang}  established a novel variational approach to study the existence of positive
solutions of \eqref{1.2} via $$\frac{t^{-\alpha}-1}\alpha\to\ln\frac1t\quad\mathrm{as~}\alpha\to0.$$
For more related results on \eqref{1.2}, one can refer to \cite{cao-dai,Chen-Tang,Cingolani-Weth,du,Liu-Radulescu-Zhang,wangli} and references therein.

 We point out that  in the higher dimensional case $N\geq 3$, Wang and Yi \cite{Wang-Yi1} established the uniqueness of positive radial  ground state solutions to \eqref{1.1} by developing Lieb's method. Later, Guo, Wang and Yi   \cite{Guo-Wang-Yi} proved the existence, radial symmetry, uniqueness of the positive
 ground state solution   of \eqref{model3}  on  a ball $B_R\subset \mathbb{R}^N$ with $N\geq 3$ under the Dirichlet boundary condition, that is,
\begin{equation}\label{model4}
\left\{
\begin{array}{lll}
-\Delta u+u=wu \, &\mbox{in } B_R,\\
-\Delta w=|u|^2 \, &\mbox{in } B_R,\\
w=u=0\,  &\mbox{in }  \partial B_R.
\end{array}\right.
\end{equation}
By using Green's function again, \eqref{model4}
can  be rewritten as the following Choquard type equation on a ball
\begin{equation}\label{model5}
-\Delta u+u=\left(\int_{B_R}G(x,y)|u^2(y)|dy\right)u,\ \ x\in B_R.
\end{equation}
Here
\begin{equation}\label{gdingyi}
G(x,y)=\left\{
\begin{array}{lll}
&\frac{1}{|x-y|^{N-2}}-\frac{1}{\left|\frac{R x}{|x|}-\frac{|x|y}{R}\right|)^{N-2}},\, N\geq3,\\
&\ln\frac{1}{|x-y|}-\ln\frac{1}{\left|\frac{R x}{|x|}-\frac{|x|y}{R}\right|},\, N=2
\end{array}\right.
\end{equation}
for any $x,y\in B_R$ with $ x\neq y$.
In addition,  the authors in \cite{Guo-Wang-Yi} showed  that as the radius $R\to\infty$, the unique positive ground state solution of \eqref{model5} converges to the unique positive ground state solution of \eqref{1.1}. Now a natural question arises:
can we find such qualitative properties including  radial symmetry, uniqueness and convergence in two dimensional case
even though the   logarithmic kernel is sign-changing and unbounded both from above and below? In this paper, we shall give an affirmative  answer.

In this paper,  we consider the existence and qualitative properties of  \eqref{model5}  on the disc  $D_R\subset \mathbb{R}^2$ with radius $R>0$ centered at the origin, that is,
\begin{equation}\label{*1}\left\{\begin{aligned}
-\Delta u+u&=\left(\int_{D_R}\big(\ln\frac{1}{|x-y|}-\ln\frac{1}{\left|\frac{R x}{|x|}-\frac{|x|y}{R}\right|}\big)u^2(y)dy\right) u\quad \mbox{in}\,D_R,\\
u&=0\quad  \mbox{on}\,\ \partial D_R,
\end{aligned}\right.\end{equation}
Here we remark that  the  existence and uniqueness of positive radially symmetric solution to  \eqref{1.2} in $\mathbb{R}^2$  have been obtained in  \cite{Cingolani-Weth}.
Hereafter,  we say $u$ is a ground state solution of an equation if $u$ solves   \eqref{*1}   and minimizes the energy functional associated  with  \eqref{*1} among all possible nontrivial solutions.

 Our main results are as follows.
\begin{thm}\label{theorem1}
All the positive  ground state solutions  of  \eqref{*1} are radially symmetric and decreasing.
\end{thm}
\begin{thm}\label{th4.1}
The    positive  ground state solution  of  \eqref{*1} is uniquely determined.
\end{thm}
\begin{thm}\label{theorem3}
The    positive  ground state solution  of  \eqref{*1}  converges  to the trivial solution  as    $R\to\infty$.
\end{thm}

Compared with  the higher dimensional case, the appearance of  the logarithmic convolution  term makes equation \eqref{*1} extremely harder to handle than the Riesz convolution, because the logarithmic convolution is sign-changing and  neither bounded from above nor from below.  In view of the subspace $X$ in \cite{Stubbe}, it is natural for us to think that
a smaller  subspace than $H_0^1(D_R)$ seems necessary for  \eqref{*1} in order to obtain the existence of the ground state solutions of  \eqref{*1}. But to our surprise,   equation \eqref{*1} has some    different properties from the higher dimension case \cite{Guo-Wang-Yi}, especially the convergence property.

This paper is organized as follows. In Section 2, some preliminaries are presented and a variational framework in $H_0^1(D_R)$  is introduced by establishing some subtle inequalities. In Section 3, we prove  the radial symmetry of all the positive ground state solutions by using the rearrangement inequalities. In Section 4, we develop the Newton's Theorem for the two dimensional case and then combine the contraction mapping principle to obtain  the uniqueness of the positive ground state solutions to \eqref{*1}. Finally in Section 5, we show that the unique positive ground state solution of \eqref{*1} converges to the trivial solution of \eqref{1.2} as $R\to\infty$, instead of the unique positive ground state solution of  \eqref{1.2}, which is different from the higher dimension case.

\section{Preliminaries}
In this section, we first give some notations.
Let  $H^1(\mathbb{R}^2)$ be the Sobolev space with standard norm
$\|u\|=(\int_{\mathbb{R}^2}|\nabla u|^2+|u|^2)^{\frac{1}{2}}$.
We identify $u\in H_0^1(D_R)$ with its zero extension to $\mathbb{R}^2$ by setting $u=0$ in $\mathbb{R}^N\backslash D_R.$
 The dual space of  $H^{1}(\mathbb{R}^2)$ is denoted by $H^{-1}(\mathbb{R}^2)$.
Let $\langle\cdot,\cdot\rangle$ be the duality pairing between $H^{1}(\mathbb{R}^2)$ and $H^{-1}(\mathbb{R}^2).$ For $1\leq s<\infty$, $L^s(D_R)$ denotes the Lebesgue space with the norm
  $\|u\|_{L^s(D_R)}=\left(\int_{D_R} |u|^sdx \right)^{\frac{1}{s}}.$ Throughout the whole paper, $C$ may  represent  different  positive constants.

Next, we have the following lemma.
\begin{lemma}\label{lemma1}
 For each $u\in H_0^1(D_R)$, there holds that
 \begin{equation}\label{gxingzhi0}
 \int_{D_R}\int_{D_R}G(x,y)|u(x)|^2|u(y)|^2dxdy<+\infty.
 \end{equation}
\end{lemma}
\begin{proof}
According to  the definition  $G(x,y)$ (see \eqref{gdingyi}),  we have   $G(x,y)>0$ and $G(x,y)=G(y,x)$ for any $x,y\in D_R, x\neq y.$
Then  direct computations give that
\begin{equation}\label{gxingzhi}\begin{aligned}
0<G(x,y)& =\ln\frac{\sqrt{R^{4}+x^{2}y^{2}-2R^{2}x\cdot y}}{R|x-y|}
&\leq\ln\left(\frac{R^2}{|x-y|}+\frac{\sqrt{|x^{2}y^{2}-2R^{2}x\cdot y|}}{R|x-y|}\right)  \\
&\leq\ln(\frac{R^2}{|x-y|}+1) \leq\frac{CR^2}{|x-y|},
\end{aligned}\end{equation}
due to the fact that  $\ln(1+\frac{1}{r})\leq \frac{C}{r}$ for any $r>0.$ By
Hardy-Littlewood-Sobolev inequality, it follows immediately that
for any $u\in H_0^1(D_R),$
\begin{equation}\label{gxingzhi1}\begin{aligned}
\int_{D_R}\int_{D_R} \frac{1}{|x-y|}|u(x)|^2|u(y)|^2dxdy
\leq C\|u^{2}\|_{L^{\frac{4}{3}}(D_R)}^{2}
\leq C\|u\|^{4}<+\infty.
\end{aligned}\end{equation}
Thus \eqref{gxingzhi0} follows easily from \eqref{gxingzhi} and \eqref{gxingzhi1}.
\end{proof}

Based on   Lemma \ref{lemma1}, the energy functional $I_R:H_0^1(D_R)\to\mathbb{R}$ associated with equation \eqref{*1} is given by
$$
\begin{aligned}I_R(u)&=\frac{1}{2}\int_{D_R}(|\nabla u|^2+|u|^2)dx-\frac{1}{4}\int_{D_R}\int_{D_R}G(x,y)|u(y)|^2|u(x)|^2dxdy,
\end{aligned}
$$
which is well defined on $H_0^1(D_R)$.
In addition, it is easy to check that $I_R$ is $C^1$ and its Gateaux derivative is
$$\begin{aligned}
\langle I'_R(u),v\rangle&=\int_{D_R}(\nabla u\nabla v+uv)dx-\int_{D_R}\int_{D_R}G(x,y)|u(y)|^2u(x)v(x)dxdy
\end{aligned}$$
for any $v\in H_0^1(D_R)$. So the critical points of $I_R$ are solutions of \eqref{*1} in the weak sense.

\section{Existence and radial symmetry}
In this section,  we shall prove the existence  and radial symmetry of ground state solutions  of  \eqref{*1} by resorting to the symmetric decreasing rearrangement and  Talenti's inequality. First recall the following Talenti's inequality  (see \cite[Theorem 3]{Alvino-Lion}).
\begin{lemma}\label{lemma6}
Let $0\leq f\in L^2(D_R)$, and let $u,v\in H^1_0(D_R)$ solve
\begin{equation*}\left\{
\begin{array}{lll}
-\Delta u=f, \ x\in D_R,\\
u=0,\ x\in \partial D_R,
\end{array}\right.
\end{equation*}
and
\begin{equation*}\left\{
\begin{array}{lll}
-\Delta v=f^*, \ x\in D_R,\\
v=0,\ x\in \partial D_R.
\end{array}\right.
\end{equation*}
Then $u^*\leq v$ a.e. in $D_R$. If additionally $u^*(x_0)=v(x_0)$ for
some $x_0$ with $|x_0|=t\in (0,R),$ then $u(x)=v(x)$ and $f(x)=f^*(x)$ for
all $x$ with $t\leq|x|\leq R.$ Here $u^*$ and $f^*$ are the symmetric decreasing rearrangement of $u$ and $f$, respectively.
\end{lemma}

As usual, we define the Nehari manifold
$$\mathcal{N}_R=\{u\in H_0^1(D_R)\backslash\{0\}:\langle I_R^{\prime}(u),u\rangle=0\}.$$
For each $u\in H_0^1(D_R)\backslash\{0\},$ there exists a unique $t_u\in(0,\infty)$ such that $t_u u\in\mathcal{N}_R.$ So we can deduce that $\mathcal{N}_R\neq \emptyset$
and the following conclusion is true.
\begin{prop}\label{thmexistence}
The infimum number
$$c_R:=\inf_{\phi\in\mathcal{N}_R} I_R(\phi)=\inf_{\phi\in H_0^1(D_R)\setminus\{0\}}\sup_{t>0}I_R(t \phi)$$
is achieved by a function $\phi_R\in H_0^1(D_R)$ which is a ground state solution of  \eqref{*1}. Moreover,
  $\phi_R\in C^2(D_R)\cap H_0^1(D_R)$ is strictly positive, radially symmetric and decreasing.
\end{prop}
\begin{proof}
Observe that
$$I_R(t_uu)=\sup\limits_{t>0}I_R(tu)=\frac14\frac{\|u\|^4}{\int_{D_R}\int_{D_R}G(x,y)|u(y)|^2|u(x)|^2dxdy}$$
and thereby
$$c_R=\inf_{u\in\mathcal{N}_R}I_R(u)=\inf_{u\in H_0^1(D_R)\setminus\{0\}}\sup_{t>0}I_R(tu).$$
Then by using   standard Nehari manifold method as in \cite[Chapter 4]{Willem}, we can obtain the existence of ground state solutions of \eqref{*1} in $H_0^1(D_R)$.

 We assume that $\phi_R\in H_0^1(D_R)$  is a ground state solution of \eqref{*1}. It suffices to prove the positivity, radial symmetry and decreasing of $\phi_R$.
Notice that $|\phi_R|\in \mathcal{N}_R$ and $I_R(\phi_R)=I_R(|\phi_R|)$. So $|\phi_R|$ is also a ground state solution of \eqref{*1}. Then by applying the standard elliptic regularity theory and the strong maximum principle, we see that $\phi_R\in C^2(D_R)$, and either $\phi_R> $ or $\phi_R< 0$ in $D_R$. Without loss of generality, we assume that $\phi_R>0$ in $D_R$.

Let $\phi_R^*$ be the symmetric decreasing rearrangement of $\phi_R.$ According to Lemma \ref{lemma6} and the  symmetric rearrangement inequalities \cite{Lieb}, we have $\int_{D_R}|\nabla\phi_R^*|^2dx \leq\int_{D_R}|\nabla\phi_R|^2dx$, $\int_{D_R}|\phi_R^*|^2dx=\int_{D_R}|\phi_R|^2dx$ and
\begin{equation}\label{juanjidengshi}
\int_{D_R}\int_{D_R}G(x,y)|\phi_R^*(y)|^2|\phi_R^*(x)|^2dxdy\geq \int_{D_R}\int_{D_R}G(x,y)|\phi_R(y)|^2|\phi_R(x)|^2dxdy.
\end{equation}
This yields that
$$ I_R(\phi_R)\geq I_R(\phi_R^*)\quad\mbox{and}\quad \langle I_R^{\prime}(\phi_R^*),\phi_R^*\rangle\leq 0.$$

We claim that
$\langle I_R^{\prime}(\phi_R^*),\phi_R^*\rangle=0$, that is, $\phi_R^*\in\mathcal{N}_R.$
Otherwise, $\langle I_R^{\prime}(\phi_R^*),\phi_R^*\rangle<0$ and thereby there exists $t_{\phi_R^*}\in(0,1)$ such that $t_{\phi_R^* }\phi_R^* \in \mathcal{N} _R.$
Then
$$\begin{aligned}
c_R\leq I_R(t_{\phi_R^*}\phi_R^*)
&=\frac{1}{4}\frac{\|\phi_R^*\|^4}{\int_{D_R}\int_{D_R}G(x,y)|\phi_R^*(y)|^2|\phi_R^*(x)|^2dxdy}\\
&<\frac{1}{4}\frac{\|\phi_R\|^4}{\int_{D_R}\int_{D_R}G(x,y)|\phi_R(y)|^2|\phi_R(x)|^2dxdy}=I_R(\phi_R)=c_R,
\end{aligned}$$
which is a contradiction. The claim  holds.
Hence, $\phi_R^*$ is a positive ground state solution of \eqref{*1} and $$\int\limits_{D_R}\int\limits_{D_R}G(x,y)|\phi_R^*(y)|^2|\phi_R^*(x)|^2dxdy= \int_{D_R}\int_{D_R}G(x,y)|\phi_R(y)|^2|\phi_R(x)|^2dxdy.$$ This together with  Lemma \ref{lemma6} yields that
$\phi_R^*=\phi_R.$
Thus $\phi_R\in C^2(D_R)\bigcap H_0^1(D_R)$ is a strictly positive, radially symmetric and decreasing. The proof is finished.
\end{proof}
 \noindent {\textbf{Proof of Theorem \ref{theorem1}}}: It is a direct conclusion of  Proposition  \ref{thmexistence}.  \qed

\section{Uniqueness}
Based on the radial symmetry obtained in Theorem \ref{thmexistence}, we are going to prove the uniqueness of the positive ground state solution of \eqref{*1} in this section. From now on, we always assume that $\phi_R$ is a positive radial ground state solution of  \eqref{*1}.

In order to achieve our goal,
we first develop the Newton's theorem for the two dimensional case. Let $S^1$ denote the unit circle in $\mathbb{R}^2$ centered at the origin and ${\mathbb{S}^1_{r,-x}}$ be a circle with radius $r>0$ centered at the point $-x$ in $\mathbb{R}^2$.
\begin{lemma}\label{lemniud}
  For any $x\in\mathbb{R}^2$ and $r>0$,
\begin{equation}\label{niudundenshi}
P(r,x):=\frac{1}{2\pi}\int_{\mathcal{S}^1}\ln \frac{1}{|rz-x|}dz=\min\left\{\ln \frac{1}{r},\ln \frac{1}{|x|}\right\}.
\end{equation}
\end{lemma}
\begin{proof}
  Observe that $P(r,x)$ is radial with respect to $x$. Since $\ln\frac{1}{|x|}$  is a
harmonic function if $x\neq 0$,  by using the mean-value formula, we deduce that
$\frac{1}{2\pi}\int_{\mathcal{S}^1}\ln \frac{1}{|rz-x|}dz=\ln \frac{1}{|r|}$ if $r>|x|$, and $\frac{1}{2\pi}\int_{\mathcal{S}^1}\ln \frac{1}{|rz-x|}dz=\ln \frac{1}{|x|}$ if $r<|x|$.
It suffices to prove the case $r=|x|. $ Notice that $P(r,x)=\frac{1}{2\pi}\int_{\mathbb{S}^1_{r,-x}}\ln \frac{1}{|y|}dy$.
By using polar coordinate formula, there holds that
$\lim\limits_{{\varepsilon}\to0}\frac{1}{2\pi}\int_{\mathbb{S}^1_{r,-x}\bigcap D_{\varepsilon}}\ln \frac{1}{|y|}dy= 0.$
This,  combined with Lebesgue dominated convergence theorem implies that $P(r,x)=\ln\frac{1}{|x|}$. The proof is finished.
\end{proof}

With the help of Lemma \ref{lemniud}, we have the following result, which is a novel point in this paper.
\begin{lemma}\label{lemjuanji}
For any radial function $\varphi\in H^1_0(D_R)$ and $x\in D_R,$ we have
\begin{equation}\label{juanji}\begin{aligned}
 \int_{D_R} G(x,y)|\varphi(y)|^2dy=\int_{|y|\leq |x|}\left(\ln \frac{|y|}{ |x|}\right)|\varphi(y)|^2dy+\int_{D_R}\left(\ln \frac{R}{ |y|}\right)|\varphi(y)|^2dy.
\end{aligned}\end{equation}
\end{lemma}
\begin{proof}
First we can check that $\int_{D_R}\left(\ln \frac{1}{ |x-y|}\right)|\varphi(y)|^2dy$ is   a radial function about $x$ due to the fact that  $\varphi\in H^1_0(D_R)$ is   radial. Then by  applying Lemma \ref{lemniud}, we obtain that

\begin{equation*}\begin{aligned}
\int_{D_R}|\varphi(y)|^2\ln \frac{1}{ |x-y|}dy&=\frac{1}{2\pi}\int_{\mathcal{S}^1}\left[\int_{D_R}|\varphi(y)|^2\ln \frac{1}{ ||x|z-y|}dy\right]dz\\
&=\int_{D_R}\left[\frac{1}{2\pi}\int_{\mathcal{S}^1}\left(\ln \frac{1}{ ||x|z-y|}\right)dz\right]|\varphi(y)|^2dy\\
&=\int_{D_R}\min\left\{\ln\frac{1}{|x|},\ln\frac{1}{|y|}\right\}|\varphi(y)|^2dy\\
&=\int_{|y|\leq |x|}\left(\ln \frac{1}{ |x|}\right)|\varphi(y)|^2dy+\int_{|y|>|x|}\left(\ln \frac{1}{|y|}\right)|\varphi(y)|^2dy.\\
&=\int_{|y|\leq |x|}\left(\ln \frac{1}{ |x|}-\ln \frac{1}{|y|}\right)|\varphi(y)|^2dy+\int_{D_R}\left(\ln \frac{1}{ |y|}\right)|\varphi(y)|^2dy
\end{aligned}\end{equation*}
and
\begin{equation*}\label{gjifen}\begin{aligned}
\int_{D_R}|\varphi(y)|^2\ln\frac{1}{\left|\frac{R x}{|x|}-\frac{|x|y}{R}\right|}dy
=&\int_{D_R}|\varphi(y)|^2\ln\frac{1}{\left|\frac{R y}{|y|}-\frac{|y|x}{R}\right|}dy\\
=&\int_{D_R}|\varphi(y)|^2\ln \frac{\frac{R}{|y|}}{\left|x-\frac{R^2}{|y|^2}y\right|}dy,  \\
=&\int_{D_R}|\varphi(y)|^2\left[\frac{1}{2\pi}\int_{S^1}\ln \frac{\frac{R}{|y|}}{\left||x|z-\frac{R^2}{|y|^2}y\right|}dz\right]dy\\
=&\int_{D_R}|\varphi(y)|^2 \min\left\{\ln\frac{R}{|xy|},\ln \frac{1}{R}\right\}dy\\
=&\int_{|x|\leq\frac{R^2}{|y|}}|\varphi(y)|^2\ln \frac{1}{R}dy +\int_{|x|>\frac{R^2}{|y|}}|\varphi(y)|^2\ln \frac{R}{|xy|}dy\\
=&\int_{D_R}|\varphi(y)|^2\ln \frac{1}{R}dy.
\end{aligned}\end{equation*}
Thus \eqref{juanji} follows immediately.
\end{proof}
For the sake of  convenience, we denote by
$$U_{\phi}(x)=\int_{|y|\leq|x|}\left(\ln\frac{|x|}{|y|}\right)|\phi(y)|^2dy.$$
 We can see that $U_{\phi_R}>0$ is radial, and by using the polar coordinates formula,  $$U_{\phi_R}(x)=\int_0^{|x|}2\pi r\left(\ln\frac{|x|}{r}\right)|\phi_R(r)|^2dr<+\infty.$$
Then according to  Lemma \ref{lemjuanji}, equation \eqref{*1} becomes
\begin{equation}\label{bianhuanfangc}
\left\{\begin{aligned}
-\Delta \phi_R+U_{\phi_R}\phi_R&=\left(\int_{D_R}\left(\ln \frac{R}{ |y|}\right)|\phi_R(y)|^2dy-1\right)\phi_R,  &\mbox{in}\  D_R,\hfill\\
\phi_R&=0,  &\mbox{on}\  \partial D_R.\hfill
\end{aligned}\right.
\end{equation}
Hence it suffices   to prove the uniqueness of the positive ground state solution for \eqref{bianhuanfangc}. To this end, we denote by
\begin{equation}\label{20}
\lambda(\phi_R)=\int_{D_R}\left(\ln \frac{R}{ |y|}\right)|\phi_R(y)|^2dy-1.
\end{equation}
Clearly, $\lambda(\phi_R) \in(0,\infty).$ By taking
$$\mbox{$R^*:=\sqrt{\lambda}R$ \quad and \quad $\varphi_{\lambda}(x)=\frac{1}{\lambda}\phi_R(\frac{x}{\sqrt{\lambda}})$}$$
with $\lambda=\lambda(\phi_R),$ we conclude that $\varphi_{\lambda}$ satisfies
\begin{equation}\label{*2}
\left\{\begin{aligned}
-\Delta \phi+U_{\phi}\phi&=\phi\quad \mbox{in}\,\ D_{R^{\ast}},\hfill\\
\phi&=0\quad \mbox{on}\,\ \partial D_{R^{\ast}}.\hfill
\end{aligned}\right.
\end{equation}

In the following, we first investigate the uniqueness of positive ground state radial solutions of \eqref{*2}. Assume that $\phi\in C^{2}(D_{R^*})\cap H_0^1(D_{R^*})$  is a positive  ground state radial solution of \eqref{*2}. Let
$$\Gamma_{\phi}=\inf \{A_{\phi}(\psi):\|\psi\|_{L^2(D_{R^*})} =1\},$$
where  $A_{\phi}:H_0^1(D_{R^*})\to\mathbb{R}$ is given by
$$A_{\phi}(\psi)=\int_{D_{R^*}}(|\nabla\psi|^2+U_\phi \psi^2)dx.$$

\begin{lemma}\label{lem4.5}
 $\Gamma_{\phi}$ can be achieved by a positive radial function $\hat{\psi}\in H_0^1(D_{R^*})$ and $\Gamma_\phi=1.$
\end{lemma}
\begin{proof}
By applying similar arguments as in  \cite[Lemma 4.5]{Guo-Wang-Yi}, the conclusion is true.
\end{proof}
\begin{lemma}\label{lem4.6}
The positive radial solution of \eqref{*2} is unique.
\end{lemma}
\begin{proof}
We shall prove it by contradiction. Suppose on the contrary that there are two different positive radial solutions
$\phi_{1}, \phi_{2}\in C^{2}(D_{R^{\ast}})\cap H_0^1(D_{R^*})$ with $\phi_{1}\neq\phi_{2}.$ Then $\phi_i$ with $i=1,2$ satisfies
\begin{equation}\label{*3}
\left\{\begin{aligned}
&\phi_i''+\frac{1}{r}\phi_i'=(U_{\phi_i}-1)\phi_i,\quad r\in[0,R^{\ast}),\\
&\phi_i(R^*)=0,\quad \phi_i'(0)=0.
\end{aligned}\right.
\end{equation}
Let
$$\psi=\phi_{1}-\phi_{2}.$$
Then $\psi\nequiv 0$  and there are three cases may happen:
\begin{description}
  \item[(i)] either $\psi\geq0 \ \ in \ \ [0,R^{\ast})$ or $\psi\leq0 \ \ in \ \ [0,R^{\ast})$;
  \item[(ii)] there is $R_1\in(0,R^*)$ such that $\psi(R_{1})=0$, and either $\psi\geq 0$ or $\psi\leq 0$ in $[0,R_{1}]$;
  \item[(iii)] there is $R_2\in(0,R^*)$ such that $\psi\equiv 0$ in $[0,R_{2}]$ and $\psi$ changes sign in $(R _{2},R_{2}+\varepsilon)$ for any $\varepsilon>0$.
\end{description}

If \textbf{(i)} happens, we assume that $\psi\geq 0$ in $[0,R^{\ast})$. The other case $\psi\leq0$ can be discussed  similarly. Notice that
$U_{\phi_{2}}<U_{\phi_{1}}$
and by Lemma \ref{lem4.5}, we have
$$A_{\phi_{1}}(\varphi)\geq \|\varphi\|^{2}_{L^{2}(D_{R^{\ast}})}\quad \mbox{ and }\quad \ A_{\phi_{2}}(\varphi)\geq \|\varphi\|^{2}_{L^{2}(D_{R^{\ast}})}\quad \mbox{for all}\  \varphi \in H_{0}^{1}(D_{R^{\ast}}).$$
Hence,
$$ \|\phi_{1}\|_{L^{2}(D_{R^{\ast}})}^{2}\leq A_{\phi_{2}}(\phi_{1})=A_{\phi_{1}}(\phi_{1})+\int_{D_{R^{\ast}}}(U_{\phi_{2}}-U_{\phi_{1}})\phi_{1}^{2}<\|\phi_{1}\|_{L^{2}(D_{R^{\ast}})}^{2},$$
which leads a contradiction. Hence \textbf{(i)} will not happen.

If \textbf{(ii)} happens, set
\begin{equation*}
\tilde{\psi}(x)=\left\{\begin{aligned}
\psi(x),\quad &\mbox{if}\,\ x\in D_{R_{1}},\\
0,\quad &\mbox{if}\, x\in D_{R^{\ast}}\backslash D_{R_{1}}.
\end{aligned}\right.\end{equation*}
Then   by direct computations, we deduce  from \eqref{*2} that $$\left(-\Delta+\frac{1}{2}(U_{\phi_{1}}+U_{\phi_{2}})\right)\tilde{\psi}=\tilde{\psi}-\frac{U_{\phi_{1}}-U_{\phi_{2}}}{2}(\phi_{1}+\phi_{2}) \ \ \mbox{in }\ \ D_{R_{1}}.$$
Multiplying both sides by $\tilde{\psi}$ and integrating over $D_{R^{\ast}}$, we have
$$\|\tilde{\psi}\|^{2}_{L^{2}(D_{R^{\ast}})}\leq
\frac{1}{2}A_{\phi_{1}}(\tilde{\psi})+\frac{1}{2}A_{\phi_{2}}(\tilde{\psi})=\|\tilde{\psi}\|^{2}_{L^{2}(D_{R^{\ast}})}-\int_{D_{R^{\ast}}}
\frac{U_{\phi_{1}}-U_{\phi_{2}}}{2}(\phi_{1}+\phi_{2})\tilde{\psi}
<\|\tilde{\psi}\|^{2}_{L^{2}(D_{R^{\ast}})}.
$$
This leads a contradiction. Hence \textbf{(ii)} will not happen.

If \textbf{(iii)} happens, we obtain
$\phi_1(R_2)=\phi_2(R_2), \ \phi_1^{'}(R_2)=\phi_2^{'}(R_2).$ Then by applying the variation of constants formula to \eqref{*3}, we have
\begin{equation}\label{jifendengshi}
\phi_{1}(r)-\phi_{2}(r)=T(r,\phi_{1})-T(r,\phi_{2}),\quad \mbox{for any } r\in(R_2,R^*),
\end{equation}
where $$T(r,\phi_{i})=\ln r\int_{R_{2}}^{r}s(U_{\phi_{i}}(s)-1)\phi_{i}(s)ds+\int_{R_{2}}^{r}(\ln\frac{1}{s}) s(U_{\phi_{i}}(s)-1)\phi_{i}(s)ds.$$
Observe that for any $r\in(R_2,R_2+\varepsilon)$,
\begin{equation}\label{uxing1}\begin{aligned}
U_{\phi_i}(r)=2\pi\int_{0}^{r}\left(\ln\frac{r}{s}\right)s|\phi_i(s)|^2ds\leq Cr
\end{aligned}\end{equation}
and
\begin{equation}\label{uxing2}\begin{aligned} |U_{\phi_{1}}(r)-U_{\phi_{2}}(r)|
&\leq2\pi\int_{0}^{r}\left(\ln\frac{r}{s}\right)s|\phi_{1}(s)-\phi_{2}(s)||\phi_{1}(s)+\phi_{2}(s)|ds\\
&=2\pi\int_{R_{2}}^{r}\left(\ln\frac{r}{s}\right)s|\phi_{1}(s)-\phi_{2}(s)||\phi_{1}(s)+\phi_{2}(s)|ds\\
&\leq C \sup_{(R_{2},R_{2}+\varepsilon)}|\phi_{1}(s)-\phi_{2}(s)|
\end{aligned}\end{equation}
for some constant $C>0.$ Then we conclude from \eqref{jifendengshi}-\eqref{uxing2} that for $\epsilon>0$ sufficiently small, there is $C_{\epsilon}\in(0,1/2)$ such that for any $r\in(R_2,R_2+\varepsilon)$,
$$\begin{aligned}
&|\phi_{1}(r)-\phi_{2}(r)|=|T(r,\phi_{1})-T(r,\phi_{2})|\\
\leq&\left|\int_{R_{2}}^{r}(\ln \frac{1}{s}) s\left(U_{\phi_{1}}\phi_{1}-U_{\phi_{2}}\phi_{2}+\phi_{2}-\phi_{1}\right)ds\right|+
\left|\ln r\int_{R_{2}}^{r}s\left(U_{\phi_{1}}\phi_{1}-U_{\phi_{2}}\phi_{2}+\phi_{2}-\phi_{1}\right)ds\right|\\
=&\left|\int_{R_{2}}^{r}(\ln \frac{1}{s}) s\left(U_{\phi_{1}}(\phi_{1}-\phi_{2})+(U_{\phi_{1}}-U_{\phi_{2}})\phi_{2}+\phi_{2}-\phi_{1}\right)ds\right|\\
&+\left|\int_{R_{2}}^{r}(\ln  {r}) s\left(U_{\phi_{1}}(\phi_{1}-\phi_{2})+(U_{\phi_{1}}-U_{\phi_{2}})\phi_{2}+\phi_{2}-\phi_{1}\right)ds\right|\\
\leq& C_{\varepsilon} \sup_{(R_{2},R_{2}+\varepsilon)}|\phi_{1}(s)-\phi_{2}(s)|.
\end{aligned}$$
This leads to a contradiction with $\psi\nequiv 0$ in $(R_2,R_2+\epsilon).$  Hence \textbf{(iii)} will not happen.

From all the above arguments,  the proof is completed.
\end{proof}

\noindent {\textbf{Proof of Theorem \ref{th4.1}}}:
Suppose that \eqref{*1} admits two positive radial ground state solutions $\varphi_1,\varphi_2\in H_0^1(D_R).$
Then either
\begin{description}
  \item[Case 1:] $\lambda(\varphi_1)=\lambda(\varphi_2)$ or
  \item[Case 2:] $\lambda(\varphi_1)\neq\lambda(\varphi_2).$
\end{description}

If \textbf{Case 1} happens, then by the scaling, \eqref{*1} is reduced to \eqref{*2}. According to Lemma \ref{lem4.6}, we obtain  $\varphi_1\equiv \varphi_2$. So the proof is finished.

If \textbf{Case 2} happens, without loss of generality, we may assume $\lambda(\varphi_1)>\lambda(\varphi_2).$ The other case $\lambda(\varphi_1)<\lambda(\varphi_2)$ can be treated similarly.  We set
$\tilde{\lambda}=\frac{\lambda(\varphi_1)}{\lambda(\varphi_2)}>1$
and $\tilde{\varphi}_2(x)=\tilde{\lambda}\varphi_2\left(\sqrt{\tilde{\lambda}}x\right).$ Then $\tilde{\varphi}_2$ satisfies
\begin{equation*}\left\{\begin{aligned}
-\Delta \tilde{\varphi}_2+U_{\tilde{\varphi}_2}\tilde{\varphi}_2
&=\lambda(\varphi_1)\tilde{\varphi}_2\,&\mbox{ in } D_{R/\sqrt{\tilde{\lambda}}},\\
\tilde{\varphi}_2&=0\,&\mbox{ on } \partial D_{R/\sqrt{\tilde{\lambda}}}.
\end{aligned}\right.\end{equation*}
In other words, $\tilde{\varphi}_2$ satisfies
\begin{equation}\label{tldf2}\left\{\begin{aligned}
&-\tilde{\varphi}_2''-\frac{1}{r}\tilde{\varphi}_2'+U_{\tilde{\varphi}_2}\tilde{\varphi}_2
=\lambda(\varphi_1)\tilde{\varphi}_2\,\mbox{ for } r\in(0,R/\sqrt{\tilde{\lambda}}),\\
&\tilde{\varphi}_2(R/\sqrt{\tilde{\lambda}})=0\mbox{ and } \tilde{\varphi}_2'(0)=0.
\end{aligned}\right.\end{equation}
Recall that $\varphi_1$ satisfies
\begin{equation}\label{tldf3}\left\{\begin{aligned}
&-\varphi_1''-\frac{1}{r}\varphi_1'+U_{\varphi_1}\varphi_1
=\lambda(\varphi_1)\varphi_1\,\mbox{ for } r\in(0,R),\\
&\varphi_1(R)=0\mbox{ and } \varphi_1'(0)=0.
\end{aligned}\right.\end{equation}
Then by multiplying \eqref{tldf2} and \eqref{tldf3} by $\varphi_1, \tilde{\varphi}_2$, respectively and integrating by parts, we get that for any $r\in(0,R/\sqrt{\tilde{\lambda}}),$
$$r\tilde{\varphi}'_2(r)\varphi_1(r)-\int^r_0 \tilde{\varphi}'_2s\varphi'_1 ds =\int^r_0(U_{\tilde{\varphi}_2}-\lambda(\varphi_1))\tilde{\varphi}_2\varphi_1s ds,$$
$$r\varphi'_1(r)\tilde{\varphi}_2(r)-\int^r_0 \varphi'_1s\tilde{\varphi}'_2 ds =\int^r_0(U_{\varphi_1}-\lambda(\varphi_1))\varphi_1\tilde{\varphi}_2s ds.$$
 Then
\begin{equation}\label{xiangjianshizi}
\int^r_0(U_{\varphi_1}-U_{\tilde{\varphi}_2})\varphi_1\tilde{\varphi}_2s ds=r\left(\varphi'_1(r)\tilde{\varphi}_2(r)-\tilde{\varphi}'_2(r)\varphi_1(r)\right)= r\tilde{\varphi}_2^2\left(\frac{\varphi_1}{\tilde{\varphi}_2}  \right)'.
\end{equation}

We assert that $\varphi_1(0)=\tilde{\varphi}_2(0)$.  Firstly, we show that $\varphi_1(0)<\tilde{\varphi}_2(0)$ does not happen. Otherwise, we  claim
\begin{equation}\label{duanyan1}
\varphi_1(r)<\tilde{\varphi}_2(r)\quad\mbox{ for any }\, r\in(0,R/\sqrt{\tilde{\lambda}}).
\end{equation}
Indeed, suppose on the contrary that there is $r^*\in(0,R/\sqrt{\tilde{\lambda}})$ such that $\varphi_1(r^*)=\tilde{\varphi}_2(r^*)$ and $\varphi_1(r)<\tilde{\varphi}_2(r)$ for any $r\in(0,r^*)$, that is, $\left(\frac{\varphi_1}{\tilde{\varphi}_2}  \right)(r)<1$.  Then $U_{\varphi_1}<U_{\tilde{\varphi}_2}$ in $(0,r^*)$ and thereby $\left(\frac{\varphi_1}{\tilde{\varphi}_2}  \right)'(r^*)<0$ due to \eqref{xiangjianshizi}. This implies that there exists $\delta\in (0,r^*)$  sufficiently small enough such that $\left(\frac{\varphi_1}{\tilde{\varphi}_2}\right) (r^*-\delta)>1$, which yields a contradiction. Hence the claim \eqref{duanyan1} follows. This claim contradicts with   $\varphi_1(R/\sqrt{\tilde{\lambda}})>0=\tilde{\varphi}_2(R/\sqrt{\tilde{\lambda}})$. We are done.

Secondly, we show that $\varphi_1(0)>\tilde{\varphi}_2(0)$ does not happen. In fact,  by applying  similar arguments as in \eqref{duanyan1}, we can show that $\varphi_1(r)>\tilde{\varphi}_2(r)$ for any $r\in (0,R/\sqrt{\tilde{\lambda}}).$ We regard $\tilde{\varphi}_2$ as a function in $H_0^1(D_R)$ by zero extension outside of $D_{R/\sqrt{\tilde{\lambda}}}.$ Then
\begin{equation}\label{*4}
\int_{D_R}\int_{D_R}G(x,y)|\varphi_1(x)|^2|\varphi_1(y)|^2 dxdy>\int_{D_R}\int_{D_R}G(x,y)|\tilde{\varphi}_2(x)|^2|\tilde{\varphi}^2_2(y)|^2 dxdy>0.
\end{equation}
 By \eqref{juanji} in Lemma \ref{lemjuanji}, a direct computation gives that
\begin{equation*}\begin{aligned}
&\int_{D_R}\int_{D_R}G(x,y)|\tilde{\varphi}_2(x)|^2|\tilde{\varphi}^2_2(y)|^2 dxdy\\
=&\int_{D_R}\int_{D_R}\left(\ln\frac{1}{|y-x|}\right)|\tilde{\varphi}^2_2(y)|^2|\tilde{\varphi}_2(x)|^2 dxdy-\int_{D_R}\int_{D_R}\left(\ln\frac{1}{R}\right)|\tilde{\varphi}_2(x)|^2|\tilde{\varphi}^2_2(y)|^2 dxdy\\
=&\int_{D_{R/\sqrt{\tilde{\lambda}}}}\int_{D_{R/\sqrt{\tilde{\lambda}}}}\left(\ln\frac{1}{|y-x|}\right)|\tilde{\varphi}^2_2(y)|^2|\tilde{\varphi}_2(x)|^2 dxdy-\int_{D_{R/\sqrt{\tilde{\lambda}}}}\int_{D_{R/\sqrt{\tilde{\lambda}}}}\left(\ln\frac{1}{R}\right)|\tilde{\varphi}_2(x)|^2|\tilde{\varphi}^2_2(y)|^2 dxdy\\
=&\tilde{\lambda}^{2}\int_{D_R}\int_{D_R}
\left(\ln\frac{\sqrt{\tilde{\lambda}}}{|y-x|}\right)|{\varphi}_2(y)|^2|{\varphi}_2(x)|^2dxdy -\tilde{\lambda}^{2}\left(\ln\frac{1}{R}\right)\int_{D_R}\int_{D_R}
|{\varphi}_2(y)|^2|{\varphi}_2(x)|^2 dxdy\\
=&\tilde{\lambda}^{2}\left[\int_{D_R}\int_{D_R}
\left(\ln\frac{R}{|y-x|}\right)|{\varphi}_2(y)|^2|{\varphi}_2(x)|^2dxdy+\int_{D_R}\int_{D_R}\left(\ln\sqrt{\tilde{\lambda}}\right)|{\varphi}_2(y)|^2|{\varphi}_2(x)|^2dxdy\right]\\
>&\int_{D_R}\int_{D_R}
\left(\ln\frac{R}{|y-x|}\right)|{\varphi}_2(y)|^2|{\varphi}_2(x)|^2dxdy\\
=&\int_{D_R}\int_{D_R}\left(\ln\frac{1}{|y-x|}\right)|{\varphi}_2(y)|^2|{\varphi}_2(x)|^2 dxdy-\int_{D_R}\int_{D_R}\left(\ln\frac{1}{R}\right)|{\varphi}_2(x)|^2|{\varphi}_2(y)|^2 dxdy\\
=&\int_{D_R}\int_{D_R}
G(x,y)|{\varphi}_2(y)|^2|{\varphi}_2(x)|^2dxdy.
\end{aligned}\end{equation*}
This together with \eqref{*4} yields
$$\int_{D_R}\int_{D_R}G(x,y)|\varphi_1(x)|^2|\varphi_1(y)|^2 dxdy>\int_{D_R}\int_{D_R}G(x,y)|{\varphi}_2(x)|^2|{\varphi}_2(y)|^2 dxdy.$$
However, according to Theorem \ref{thmexistence}, we have
$$c_R=\frac{1}{4}\int_{D_R}\int_{D_R}G(x,y)|\varphi_1(x)|^2|\varphi_1(y)|^2 dxdy=\frac{1}{4}\int_{D_R}\int_{D_R}G(x,y)|{\varphi}_2(x)|^2|{\varphi}_2(y)|^2 dxdy,$$ which leads to a contradiction.

Therefore, the above arguments show that  $\varphi_1(0)=\tilde{\varphi}_2(0).$
  So  $\varphi_1(0)=\varphi_2(0)$ and $\varphi'_1(0)=\varphi'_2(0)$.
  By using similar arguments as in Lemma \ref{lem4.6} case \textbf{(iii)} with some necessary modifications, we obtain that $\varphi_1\equiv \varphi_2$  and the proof is completed. \qed

\section{Convergence}
In this section, we are devoted to proving the  convergence of the unique positive ground state solution $\phi_R$ of \eqref{*1} as $R\to\infty$.
Different from the high dimensional case   $N\geq3$ (see  \cite{Guo-Wang-Yi}), we have the following lemma.
\begin{prop}\label{Prop1}
  $\lim\limits _{R\to\infty}c_R=0$.
\end{prop}
\begin{proof}
Take  a    radial cut-off function $\Psi\in C_{c}^{\infty}(D_1)$, where $D_1$ denotes the disc centered at the origin with radius $1.$  Here $\Psi$ can be regarded  as a function in $H_0^1(D_R)$ by zero extension outside of $D_1.$ Then for any $R>1$, there exists $t_{R}>0$ such that $t_{R}\Psi\in \mathcal{ N}_{R}$, and $t_{R}$ satisfies
  $$\begin{aligned}
t_{R}^{-2}\int_{D_R} (|\nabla \Psi|^{2}+|\Psi|^{2})
&= \int_{D_R}\int_{D_R}G(x,y)|\Psi(x)|^2|\Psi(y)|^2dxdy \\
&=\int_{D_R}\int_{D_R}(\ln\frac{R}{|x-y|})|\Psi(x)|^2|\Psi(y)|^2dxdy \\
&=\int_{D_1}\int_{D_1}(\ln\frac{R}{|x-y|})|\Psi(x)|^2|\Psi(y)|^2dxdy \\
&\to +\infty, \quad \mbox{ as } R\to\infty.
\end{aligned}$$
This implies $t_{R}\to 0_+$ as $R\to\infty$. Since
 $$c_{R}=\inf_{u\in \mathcal{N}_R} I_{R}(u)\leq I_{R}(t_R \Psi)=\frac{1}{4}\int_{D_R}t_{R}^{2}\left(|\nabla\Psi|^{2}+\Psi^{2}\right)=\frac{1}{4}\int_{D_1}t_{R}^{2}\left(|\nabla\Psi|^{2}+\Psi^{2}\right),$$
 it follows immediately that $\lim\limits_{R\to\infty}c_{R}=0$. The proof is finished.
\end{proof}
As a consequence of Proposition \ref{Prop1}, the following  convergence result holds.

\noindent {\textbf{Proof of Theorem \ref{theorem3}}}:
   In view of Proposition \ref{Prop1}, we have $\frac{1}{4}\int_{D_R}(|\nabla\phi_R|^2+\phi_R^2)=c_R\to0$ as $R\to \infty$. Then $\phi_R\to 0$ strongly in $H_0^1(\mathbb{R}^2)$.   Obviously, $0$ is the trivial solution of the limit equation \eqref{1.2}. The conclusion follows. \qed

\noindent \textbf{Conflict of interest.} The authors state no conflict of interest.

\noindent \textbf{Ethics approval.} No data sets were collected from human subjects during the study of the manuscript.

\noindent \textbf{Funding.} Tao Wang was supported by  the Natural Science Foundation of Hunan Province (Grant No. 2022JJ30235).
Hui Guo was supported by Scientific Research Fund of Hunan Provincial Education Department (Grant No. 22B0484) and Natural Science Foundation of Hunan Province (Grant No. 2024JJ5214).

\noindent \textbf{Data availability.} This manuscript has no associated data.

\vspace{1.0em}

\end{document}